\documentclass{amsart}
\usepackage{mathptmx}
\usepackage{xcolor}
\usepackage{hyperref}
\usepackage[utf8]{inputenc}
\usepackage[margin=2.5cm]{geometry}
\usepackage{amsaddr}
\usepackage[all]{xy}
\usepackage{amssymb}
\usepackage{amsmath}
\usepackage{enumerate}
\usepackage{graphicx}
\graphicspath{ {./images/} }
\usepackage{pdfpages}

\usepackage{systeme}
\usepackage{xcolor}

\newenvironment{declaration}{\noindent\textbf{Declaration and Statements}}{}
\newtheorem{theorem}{Theorem}[section]

\newtheorem{proposition}[theorem]{Proposition}
\newtheorem{definition}[theorem]{Definition}
\newtheorem{corollary}[theorem]{Corollary}
\newtheorem{lemma}[theorem]{Lemma}
\newtheorem{remark}[theorem]{Remark}

\newtheorem{ithm}{Theorem}[section]

\DeclareMathOperator{\Ricci}{Ric}

\def\bb#1{\mathbb{#1}}
\def\lie#1{\mathfrak{#1}}

\usepackage[all]{xy}

\newcommand{\h}{\frac{1}{2}}

\DeclareMathOperator{\Ad}{Ad}

\newcommand{\ga}{\mathsf{g}}

\makeatletter
\renewcommand{\email}[2][]{%
	\ifx\emails\@empty\relax\else{\g@addto@macro\emails{,\space}}\fi%
	\@ifnotempty{#1}{\g@addto@macro\emails{\textrm{(#1)}\space}}%
	\g@addto@macro\emails{#2}%
}
\makeatother

\title[On the dynamics of positively curved metrics on $\mathrm{SU}(3)/\mathrm{T}^2$ under the Ricci flow]{On the dynamics of positively curved metrics on $\mathrm{SU}(3)/\mathrm{T}^2$ under the homogeneous Ricci flow}
\author{Leonardo F. Cavenaghi}
\address{Instituto de Matemática, Estatística e Computação Científica, Universidade Estadual de Campinas - Unicamp\\ Rua Sérgio Buarque de Holanda, 651, 13083-859, Campinas, SP, Brazil}
\email{leonardofcavenaghi@gmail.com}

\author{Lino Grama}
\address{Instituto de Matemática, Estatística e Computação Científica, Universidade Estadual de Campinas - Unicamp\\ Rua Sérgio Buarque de Holanda, 651, 13083-859, Campinas, SP, Brazil}
\email{lino@ime.unicamp.br}

\author{Ricardo M. Martins}
\address{Instituto de Matemática, Estatística e Computação Científica, Universidade Estadual de Campinas - Unicamp\\ Rua Sérgio Buarque de Holanda, 651, 13083-859, Campinas, SP, Brazil}
\email{RMiranda@unicamp.br}

\begin{document}
\subjclass[2020]{}

\begin{abstract}
In this note, we show that the classical Wallach manifold $\mathrm{SU}(3)/\mathrm{T}^2$-admits metrics of positive intermediate Ricci curvature $(\Ricci_d >0)$ for $d = 1, 2, 3, 4, 5$ that lose these properties under the homogeneous Ricci flow for $d=1, 2, 3, 5$. We make the same analyses to the family of Riemannian flag manifolds $\mathrm{SU}(m+2p)/\mathrm{S}(\mathrm{U}(m)\times\mathrm{U}(p)\times \mathrm{U}(p))$, concluding similar results. These explicitly verify some claims expected to be true among experts (see \cite{BW}) for positive Ricci curvature and intermediate positive Ricci curvature. Our technique is only possible due to the global behavior understanding of the homogeneous Ricci flow for invariant metrics on these manifolds.
\end{abstract}
\maketitle

\section*{Declaration}
\begin{declaration}
All authors declare that they have no conflicts of interest.

Our manuscript has no associated data.
\end{declaration}

\section{Introduction}
Nowadays, an exciting research topic is guided by the search for property maintenance of geometric quantities under the \emph{Ricci flow}. Particularly to the sub-field of positive curvatures, the inaugural work of B\"ohm and Wilking \cite{BW} establishes the non-maintenance of positive Ricci curvature on a certain $12$-dimensional homogeneous space under the Ricci flow. Several works appear analyzing this and related phenomena, \cite{Bettiol2022,bamler2019,10.2307/41059321,wanshi}.

On the other hand, the Ricci flow is feasible to handle (due to better-established techniques) in the homogeneous setting, \cite{proj-ricci-flow,annalsglobal,gm12,gm09,wz}. Of particular interest is the recent paper \cite{proj-ricci-flow}, substantiated in the following:  It consists in appropriately normalizing the Ricci flow to a simplex followed by a time reparametrization to obtain polynomial equations, leading to a \emph{projected Ricci flow}. This arises from a natural generalization of the standard unit-volume reparametrization of the Ricci flow. See, for instance, Theorem \ref{thm-rescaling}. As expected by the authors of \cite{proj-ricci-flow}, such an approach works well to study the dynamics of some positively curved metrics on certain homogeneous spaces. We state our results before going a bit deeper into the machinery employed.

As the first result, Theorem \ref{thm:casobase} bellow is related to Remark 3.2 in \cite{BW}, where authors claim that there exists an invariant metric with positive Ricci curvature on the flag manifold $\mathrm{SU}(3)/\mathrm{T}^2$ that evolves under the homogeneous Ricci flow to a metric with mixed Ricci curvature. Taking this into account jointly with the fact that preservation of positive curvature conditions under \emph{backward Ricci flow} is generally not expected, it should not be taken as a surprise to have the following
\begin{ithm}[=Theorem \ref{thm:basiccase}]\label{thm:casobase}
        There exists an invariant metric on $\mathrm{SU}(3)/\mathrm{T}^2$ with positive Ricci curvature that evolves into an invariant metric of mixed Ricci curvature under the backward homogeneous Ricci flow. 
\end{ithm}
Quite recently, in \cite{davide}, the authors show the non-maintenance of positive intermediate curvatures under the homogeneous Ricci flow for a family of homogeneous spaces which encompass $M^{12} = \mathrm{Sp}(3)/\mathrm{Sp}(1)\times \mathrm{Sp}(1)\times \mathrm{Sp}(1)$, this one appearing in Theorem C in \cite{BW}. Such a work served as a particular motivation for this note. Throughout this paper, we deal with two different notions of positive intermediate Ricci curvature. As the first, we remark that the sub-index $d$ in Theorem \ref{ithm:basiccase2} stands for the concept of \emph{$d\mathrm{th}$-intermediate positive Ricci curvature}, Definition \ref{def:notiononintro0} below.

\begin{definition}\label{def:notiononintro0}
We say that a Riemannian manifold $(M,\ga)$ has positive $d\mathrm{th}$-Ricci curvature if for every $p\in M$ and every choice of non-zero $d+1$-vectors $\{v,v_1,\ldots,v_d\}$ where $\{v_1,\ldots,v_d\}$ can be completed to generated an orthonormal frame in $T_pM$, it holds $\Ricci_d(v) := \sum_{i=1}^dK_{\ga}(v,v_i) > 0$. For an $n$-dimensional manifold, these curvature conditions interpolate between positive sectional curvature ($d=1$) and positive Ricci curvature ($d=n-1$).  
\end{definition}

\begin{ithm}[=Theorem \ref{thm:basiccase2}+Proposition \ref{prop:riccinterd}]\label{ithm:basiccase2}
     For each $d = 1, 2, 3, 4$, there exists an invariant metric on $\mathrm{SU}(3)/\mathrm{T}^2$ with $\Ricci_d > 0$. Moreover, for $d = 1, 2, 3$, these evolve into invariant metrics satisfying $\Ricci_d(X) \leq 0$ for some non-zero $X$ under the homogeneous Ricci flow. 
\end{ithm}

 In \cite{wallachrelated, Abiev2016}, the authors prove that metrics with positive sectional curvature in $\mathrm{SU}(3)/\mathrm{T}^2$ evolve to metrics with planes of negative sectional curvature, thus establishing part of our Theorem \ref{thm:casobase}. Concerning Ricci curvature, \cite{wallachrelated} deals with the signature change of the Ricci tensor (Theorem 4.1). Related to it, we derive analogous results to a broader family, being our second intermediate positive Ricci curvature notion here considered:

 \begin{definition}
     Fix a  $n$-dimensional Riemannian manifold $(M,\ga)$. The Ricci tensor of $(M,\ga)$ is said to be \emph{$d$-positive} if the sum of the $d$ smallest eigenvalues of the Ricci tensor is positive at all points. This condition interpolates between positive Ricci curvature and positive scalar curvature.
 \end{definition}

\begin{ithm}[=Theorem \ref{thm:generalfamily}]\label{ithm:generalfamily}
 There exists an invariant metric on $\mathrm{SU}(m+2p)/\mathrm{S}(\mathrm{U}(m)\times \mathrm{U}(p)\times \mathrm{U}(p)), m\geq p\geq 1$ for which its Ricci tensor is $d$-positive for $d\in \left\{1,\ldots,4mp + 2p^2\right\}$. Moreover, the backward homogeneous Ricci flow loses this property in finite time.
\end{ithm}

In a very general picture, the Ricci flow for an arbitrary Riemannian manifold $(M,\ga_0)$ is the nonlinear system of
PDEs $\partial_t\ga = -2\Ricci(\ga(t)),~\ga(0) = \ga_0$. However, for homogeneous manifolds $M$, if the flow is restricted to the set of invariant metrics on $M$, this system is reduced to an autonomous nonlinear system of ODEs. Hence, it is natural and feasible to study the Ricci flow from a qualitative point of view, using tools from the dynamical systems theory.

On the one hand, the Wallach flag manifold $W^6:= \mathrm{SU}(3)/\mathrm{T}^2$ consists in an example of a homogeneous manifold that is a \emph{flag manifold of type II}, see \cite{kimura}. For our purposes in this paper, it is sufficient to face this as: the geometry of invariant metrics on $W^6$ is completely characterized by the possible choices of three summands coming from the isotropy representation on $M=G/K$. Namely, every curvature condition here treated for every invariant metric on $W^6$ admits an explicit description in terms of three positive real numbers, say $\ga = (x,y,z)$ -- Sections \ref{sec:curvatureformulae} and \ref{sec:intermediate}.

On the other hand, as it is presented in Section \ref{sec:thericciflowhere}, the Ricci flow system for invariant metrics on $W^6$ is a system of three nonlinear ODEs (in three variables), which usually would impose more difficulty on handling their dynamics. However, as introduced in \cite{proj-ricci-flow}, we can, for $W^6$ and other flag manifolds (with three isotropy summands), reduce equivalently the dynamic analyses' to the plane, where dynamical system techniques are easier to employ -- this is what is called \emph{projected homogeneous Ricci flow}.

The proof of our results is obtained by restricting the dynamics of invariant metrics $\ga = (x,y,z)$ to the tetrahedron $x+y+z=1$ via the machinery in \cite{proj-ricci-flow}, focusing particularly on the case of \emph{Riemannian submersion metrics} $\ga(t) = (t,t,1-2t)$. Precisely, we show the family $\ga(t)$ constitutes invariant solutions to the projected homogeneous Ricci flow (Lemma \ref{lem:invariantsegment}), and since it is in hand the complete behavior of the here-treated curvature formulae in terms of $t$, the proof to our results are straightforwardly obtained. Theorem \ref{ithm:generalfamily} follows via the same method.

\section{The curvature formulae of $\mathrm{SU}(3)/\mathrm{T}^2$}
\label{sec:curvatureformulae}

In this section, we derive a very explicit family of Riemannian homogeneous metrics on the Wallach flag manifold $\mathrm{SU}(3)/\mathrm{T}^2$, which have positive intermediate and positive Ricci curvature (see Definitions \ref{def:lm} and \ref{def:minimun}). Seeking this aim, we obtain explicit expressions for the sectional curvature of each basis-generated plane of an arbitrary homogeneous metric. It is worth mentioning that the derived formulas are rather classical and appear in different presentations elsewhere. For a general account of computations to the sectional curvature of some flag manifolds (Wallach manifolds), we recommend \cite{bettiol2014}. The to-be-presented formulae play a significant role in this work due to the possibility of parameterizing every curvature notion in terms of metric components here considered; see, for instance, Section \ref{sec:intermediate}.

As can be checked, the  $\mathrm{SU}(3)/\mathrm{T}^2$ consists of a \emph{flag manifold of type II} -- see \cite{kimura}. An explicit Weyl basis to $T_o\mathrm{SU}(3)/\mathrm{T}^2$ where $o = e\mathrm{T}^2$, being $e\in \mathrm{SU}(3)$ the unit element, is obtained observing that a basis for the Lie algebra of $\mathrm{SU}(3)$ is given by
$$
     \h\mathrm{diag}(2\mathrm{i},-\mathrm{i},\mathrm{i})
     , \ \h\mathrm{A}_{12}
     , \ \h\mathrm{S}_{12}
     , \ \h\mathrm{A}_{13}
     , \ \h\mathrm{S}_{13}
     , \ \h\mathrm{diag}(0,\mathrm{i},-\mathrm{i})
     , \ \h\mathrm{A}_{23}
     , \ \h\mathrm{S}_{23},
$$
where $\mathrm{S}_{kj}$ is a symmetric matrix $3\times 3$ with $\mathrm{i}$ in inputs $kj$ and $jk$ and $0$ in the others. On the other hand, $\mathrm{A}_{jk}$ is an antisymmetric matrix $3\times 3$ that has $1$ on input $kj$ and $-1$ on input $jk$, $0$ elsewhere. Moreover, $\mathrm{i} = \sqrt{-1}$. We can extract a basis for the tangent space $T_{o}\mathrm{SU}(3)/\mathrm{T}^2$ by disregarding the matrices $\mathrm{diag}(2i,-\mathrm{i}, \mathrm{i})$ and $\mathrm{diag}(0, \mathrm{i},-\mathrm{i})$. It is also worth mentioning that the $3$ components of the isotropy representation are generated by \[\mathrm{span}_{\bb R}\left\{\h\mathrm{A}_{jk},\h\mathrm{S}_{jk}\right\}.\]

 We recall that whenever a homogeneous space $M=G/K$ is \textit{reductive}, with reductive decomposition $\mathfrak{g}=\mathfrak{k}\oplus\mathfrak{m}$ (that is, $[\mathfrak{k},\mathfrak{m}]\subset \mathfrak{m}$), then $\lie m$ is $\Ad_G(K)$-invariant. Moreover, the map $\lie g \to T_o(G/K)$ that assigns to $X \in \lie g$ the induced tangent vector \[X \cdot o = \dfrac{d}{dt}\Big|_{t=0} (\exp(tX)o)\] is surjective with kernel the isotropy subalgebra $\lie k$. Using that $g \in G$ acts in tangent vectors by its differential, we have that
\begin{equation}
\label{eq-induzido}
g( X \cdot o ) = ( \Ad(g)X ) \cdot g o.
\end{equation}
Hence, the restriction $\lie m \to T_o(G/K)$ of the above map is a linear isomorphism that intertwines the isotropy representation of $K$ in $T_o(G/K)$ with the adjoint representation of $G$ restricted to $K$ in $\lie m$. This allows us to identify $T_o(G/K) = \lie m$ and the $K$-isotropy representation with the $\Ad_G(K)$-representation.

Being $G$ a compact connected simple Lie group such that the isotropy representation of $G/K$ decomposes $\mathfrak{m}$ as 
\begin{equation}\label{deco-iso}
\mathfrak{m}=\mathfrak{m}_1\oplus \ldots \oplus \mathfrak{m}_n
\end{equation}
where $\mathfrak{m}_1,\ldots,\mathfrak{m}_n$ are irreducible pairwise non-equivalent isotropy representations, all invariant metrics are given by
\begin{align}
\label{eq-compon-metr}
\mathsf{g}_o&=x_1B_1+\ldots + x_nB_n
\end{align}
where $x_i>0$ and $B_i$ is the restriction of the (negative of the) Cartan-Killing form of $\mathfrak{g}$ to $\mathfrak{m}_i$. We also have 
\begin{equation}
\label{eq-compon-ricci}
\Ricci (\mathsf{g}_o)=y_1 B_1 + \ldots + y_nB_n
\end{equation}
where 
$y_i$ is a function of $x_1, \ldots, x_n$. In our considered example, an $\mathrm{Ad}(\mathrm{T}^2)$-invariant inner product $\ga$ is determined by three parameters $(x,y,z)$ characterized by 
\begin{align*}
     \mathsf{g}\left(\h\mathrm{A}_{12},\h\mathrm{A}_{12}\right) &= \mathsf{g}\left(\h\mathrm{S}_{12},\h\mathrm{S}_{12}\right) = x,\\
     \mathsf{g}\left(\h\mathrm{A}_{13},\h\mathrm{A}_{13}\right) &= \mathsf{g}\left(\h\mathrm{S}_{13},\h\mathrm{S}_{13}\right) = y,\\
     \mathsf{g}\left(\h\mathrm{A}_{23},\h\mathrm{A}_{23}\right) &= \mathsf{g}\left(\h\mathrm{S}_{23},\h\mathrm{S}_{23}\right) = z.
\end{align*}
We then redefine new basis to $\mathfrak{m} := T_0(\mathrm{SU}(3)/\mathrm{T}^2)$ by
\[X_1 = \frac{1}{2\sqrt{x}}\mathrm{A}_{12}, X_2 = \frac{1}{2\sqrt{x}}\mathrm{S}_{12}, X_3 = \frac{1}{2\sqrt{y}}\mathrm{A}_{13}, X_4 = \frac{1}{2\sqrt{y}}\mathrm{S}_{13},\] \[X_5 = \frac{1}{2\sqrt{z}}\mathrm{A}_{23}, X_6 = \frac{1}{2\sqrt{z}}\mathrm{S}_{23} .\]

Since the following formula holds for the sectional curvature of $\ga$ (see \cite[Theorem 7.30, p. 183]{besse} or Equation 2.1 in \cite{bettiol2014})
\begin{align*}
K(X,Y) = -\frac{3}{4}\|[X,Y]_{\mathfrak{m}}\|^2 - \frac{1}{2}\mathsf{g}([X,[X,Y]_{\lie g}]_{\mathfrak{m}},Y) - \frac{1}{2}\mathsf{g}([Y,[Y,X]_{\lie g}])_{\mathfrak{m}},X)\\
+\|U(X,Y)\|^2 - \ga(U(X,X),U(Y,Y)),
\end{align*}
\begin{equation*}
2\ga(U(X,Y),Z) = \ga([Z,X]_{\lie m},Y) + \ga(X,[Z,Y]_{\lie m})
\end{equation*}
we can set up the following table, where $C_{ij}^k$ denotes a structure constant, that is, $C_{ij}^k = \ga([X_i, X_j], X_k)$, and $K_{ij}$ the sectional curvature. Moreover, that for $(i,j)\neq (1,2), (3,4), (5,6)$ it holds that
$$
    K(X_i,X_j) = K_{ij} = -\h C_{ij}^kC_{ik}^j -\h C_{kj}^iC_{ij}^k - \frac{3}{4}(C_{ij}^k)^2 + \sum_{l=1}^6\frac{1}{4}\left(C_{li}^j+C_{lj}^i\right)^2-\sum_{l=1}^6C_{li}^iC_{lj}^j.
$$
We hence build Table \ref{table:1}.

\begin{center}
\begin{table}[h!]
     \begin{tabular}{ | l | l | c | c | c |}
     \hline
     $\mathrm{i}$ & $j$ & $k$ & $C_{ij}^k$ & $K_{ij}$ \\ \hline\hline
     $1$ & $2$ & $\mathrm{diag}(\mathrm{i},-\mathrm{i},0)$ & $1/x$ & $1/x$\\ \hline
     $1$ & $3$ & $5$ & $-\frac{\sqrt{z}}{2\sqrt{xy}}$ & $-\frac{3}{16}\frac{z}{xy} + \frac{1}{8x} + \frac{1}{8y} + \frac{1}{16}\frac{(x-y)^2}{xyz}$\\ \hline
     $1$ & $4$ & $6$ & $-\frac{\sqrt{z}}{2\sqrt{xy}}$ & $-\frac{3}{16}\frac{z}{xy} + \frac{1}{8x} + \frac{1}{8y} + \frac{1}{16}\frac{(x-y)^2}{xyz}$ \\ \hline
     $1$ & $5$ & $3$ & $\frac{\sqrt{y}}{2\sqrt{xz}}$ & $-\frac{3}{16}\frac{y}{xz} + \frac{1}{8x} + \frac{1}{8z} + \frac{1}{16}\frac{(z-x)^2}{xyz}$ \\ \hline
     $1$ & $6$ & $4$ & $\frac{\sqrt{y}}{2\sqrt{xz}}$ & $-\frac{3}{16}\frac{y}{xz} + \frac{1}{8x} + \frac{1}{8z}+\frac{1}{16}\frac{(z-x)^2}{xyz}$ \\ \hline
     $2$ & $3$ & $6$ & $\frac{\sqrt{z}}{2\sqrt{xy}}$ & $-\frac{3}{16}\frac{z}{xy} + \frac{1}{8x} + \frac{1}{8y} + \frac{1}{16}\frac{(y-x)^2}{xyz}$ \\ \hline
     $2$ & $4$ & $5$ & $-\frac{\sqrt{z}}{2\sqrt{xy}}$ & $-\frac{3}{16}\frac{z}{xy} + \frac{1}{8x} + \frac{1}{8y} + \frac{1}{16}\frac{(y-x)^2}{xyz}$ \\ \hline
     $2$ & $5$ & $4$ & $\frac{\sqrt{y}}{2\sqrt{xz}}$ &   $-\frac{3}{16}\frac{y}{xz} + \frac{1}{8x} + \frac{1}{8z} + \frac{1}{8}\frac{(z-x)^2}{xyz}$ \\ \hline
     $2$ & $6$ & $3$ & $-\frac{\sqrt{y}}{2\sqrt{xz}}$ & $-\frac{3}{16}\frac{y}{xz} + \frac{1}{8x} + \frac{1}{8z} + \frac{1}{16}\frac{(z-x)^2}{xyz}$\\ \hline
     $3$ & $4$ & $\mathrm{diag}(\mathrm{i},0,-\mathrm{i})$ & $1/y$ & $1/y$ \\ \hline
     $3$ & $5$ & $1$ & $-\frac{\sqrt{x}}{2\sqrt{yz}}$ & $-\frac{3}{16}\frac{x}{yz} + \frac{1}{8y} +  \frac{1}{8z}+\frac{1}{16}\frac{(y-z)^2}{xyz}$ \\ \hline
     $3$ & $6$ & $2$ & $ \frac{\sqrt{x}}{2\sqrt{yz}}$ & $-\frac{3}{16}\frac{x}{yz} + \frac{1}{8y} +  \frac{1}{8z}+\frac{1}{16}\frac{(y-z)^2}{xyz}$ \\ \hline
     $4$ & $5$ & $2$ & $-\frac{\sqrt{x}}{2\sqrt{yz}}$ &$-\frac{3}{16}\frac{x}{yz} + \frac{1}{8y} +  \frac{1}{8z}+\frac{1}{16}\frac{(y-z)^2}{xyz}$ \\ \hline
     $4$ & $6$ & $1$ & $-\frac{\sqrt{x}}{2\sqrt{yz}}$ & $-\frac{3}{16}\frac{x}{yz} + \frac{1}{8y} +  \frac{1}{8z}+\frac{1}{16}\frac{(y-z)^2}{xyz}$ \\ \hline
     $5$ & $6$ & $ \mathrm{diag}(0,\mathrm{i},-\mathrm{i})$ & $1/z$ & $1/z$ \\ \hline
     \end{tabular}
     \caption{Structure Constants and Sectional curvature of the basis' elements \label{table:1}}
     \end{table}
     \end{center}

    In the next section, we shall recall the concept of positive intermediate Ricci curvature, further using the content in Table \ref{table:1} to provide a metric with $\Ricci_d >0$ for $d = 1, 2, 3, 4,5$ in $\mathrm{SU}(3)/\mathrm{T}^2$.
\section{Intermediate positive Ricci curvature}

\label{sec:intermediate}

\subsection{Preliminaries}

Here, we follow the good description in \cite{lm}. For the forthcoming definitions, there is no widely used notation/terminology. In particular, $d$-positivity of the Ricci tensor (Definition \ref{def:minimun}) is denoted by $\Ricci_d>0$ in \cite{crowley2020intermediate}. See \cite[Section 2.2]{DGM} or \cite[p. 5]{crowley2020intermediate} for further information.

\begin{definition}
Given a point $p$ in a Riemannian manifold $(M,\ga)$, and a collection $v,v_1,\ldots,v_{d}$ of orthonormal vectors in $T_pM$, the $d^{\mathrm{th}}$-intermediate Ricci curvature at $p$ corresponding to this choice of vectors is defined to be $\Ricci_d(v) = \sum_{i=1}^{d} K(v,v_i),$ where $K$ denotes the sectional curvature of $\ga$.
\end{definition}

\begin{definition}[$d\mathrm{th}$-intermediate positive Ricci curvature]\label{def:lm}
    We say that a Riemannian manifold $(M,\ga)$ has positive $d\mathrm{th}$-Ricci curvature if for every $p\in M$ and every choice of non-zero $d+1$-vectors $\{v,v_1,\ldots,v_d\}$ where $\{v_1,\ldots,v_d\}$ can be completed to generated an orthonormal frame in $T_pM$, it holds $\Ricci_d(v) > 0$.
\end{definition}

It is remarkable that for an $n$-dimensional manifold, these curvatures interpolate between positive sectional curvature and positive Ricci curvature for $d$ ranging between $1$ and $n-1$. Quoting \cite{lm}, the quantity presented in Definition \ref{def:lm} has been called ``$d\mathrm{th}$-intermediate Ricci curvature'', ``$d\mathrm{th}$-Ricci curvature'', ``$k$-dimensional partial Ricci curvature'', and ``$k$-mean curvature''.

Another here-considered notion of \emph{intermeidate curvature condition} is
\begin{definition}\label{def:minimun}
Let $M$ be a $n$-dimensional Riemannian manifold and let $d\leq n$. We say that the Ricci tensor of $M$ is \emph{$d$-positive} if the sum of the $d$ smallest eigenvalues of the Ricci tensor is positive at all points.
\end{definition}

\begin{remark}\label{rem:interpolates}
It is worth pointing out that if $d$ ranges from $1,\ldots,n$, the condition given by Definition \ref{def:minimun} interpolates between positive Ricci curvature and positive scalar curvature.
\end{remark}

Much work has been appearing concerning Definition \ref{def:lm}, and recent attention to this subject can be noticed. For a complete list of references on the subject, we recommend \cite{lm}. However, here we chose to explicitly cite some works we have been paying more attention to when dealing with this subject:  \cite{kennard2023positive, davide, reiser2022positive, reiser2022intermediate, Mouille2022}.

Indeed, part of the idea to this note was conceived looking to the examples approached in \cite{davide}. These were built in \cite{DGM} and provide metrics of intermediate positive Ricci curvature (in the sense of Definition \ref{def:lm}) on some \emph{generalized Wallach spaces}. In \cite{davide}, the authors show that these conditions are not preserved under the homogeneous Ricci flow. Their analyses closely follow the techniques developed in \cite{BW}.

We remark that once the Ricci curvature formula is given by 
\begin{align*}
\Ricci(X) = \sum_{\mathrm{i}=1}^nK(X,X_i),
\end{align*}
straightforward computations from Table \ref{table:1} leads to (compare with equation 2.1 in \cite{davide})
\begin{align}\label{eq:ricci01}
\Ricci(X_1) = \Ricci(X_2) =&\frac{1}{2x} + \frac{1}{12}\left(\frac{x}{y z}-\frac{z}{x y}-\frac{y}{x z}\right),\\
\Ricci(X_3) = \Ricci(X_4) =&\frac{1}{2y} + \frac{1}{12}\left(-\frac{x}{y z}-\frac{z}{x y}+\frac{y}{x z}\right) ,\\  
\Ricci(X_5) = \Ricci(X_6)=&\frac{1}{2z}+ \frac{1}{12}\left(-\frac{x}{y z}+\frac{z}{x y}-\frac{y}{x z}\right).\label{eq:ricci02}
\end{align}

Next, we derive expressions for the intermediate curvature notions presented in Definitions \ref{def:lm}, \ref{def:minimun}.

\subsection{On the intermediate positive Ricci curvatures of left-invariant metrics on $\mathrm{SU}(3)/\mathrm{T}^2$}

Our approach to deriving explicit formulae for the curvature notions introduced in the former section is combinatorial. The possibility of parameterizing all curvature notions in terms of the metric describing parameter $(x,y,z)$ plays a significant role in proving our results.

We take advantage of Table \ref{table:1} considering the symmetries appearing on the expressions for sectional curvature to get a simplified description of $d\mathrm{th}$-intermediate positive Ricci curvature (recall the Definition \ref{def:lm}). Fix $X_i$ on the basis of $\lie m$ and take $d$-vectors out of $\{X_1,\ldots,X_6\}$ for $1 \leq d\leq 5$. We have
\begin{align*}
    \Ricci_d(X_i) = a_i\frac{1}{x} + b_i\left(-\frac{3}{16}\frac{z}{xy} + \frac{1}{8x} + \frac{1}{8y}+\frac{1}{16}\frac{(x-y)^2}{xyz}\right) + c_i\left(-\frac{3}{16}\frac{y}{xz} + \frac{1}{8x} + \frac{1}{8z} + \frac{1}{16}\frac{(z-x)^2}{xyz}\right),~i=1,2,\\
    \Ricci_d(X_i) = a_i\frac{1}{y} + b_i\left(-\frac{3}{16}\frac{z}{xy} + \frac{1}{8x} + \frac{1}{8y} + \frac{1}{16}\frac{(y-x)^2}{xyz}\right) + c_i\left(-\frac{3}{16}\frac{x}{yz} + \frac{1}{8y} +  \frac{1}{8z}+\frac{1}{16}\frac{(y-z)^2}{xyz}\right),~i = 3,4,\\
    \Ricci_d(X_i) = a_i\frac{1}{z} + b_i\left(-\frac{3}{16}\frac{y}{xz} + \frac{1}{8x} + \frac{1}{8z} + \frac{1}{16}\frac{(z-x)^2}{xyz}\right) + c_i\left(-\frac{3}{16}\frac{x}{yz} + \frac{1}{8y} +  \frac{1}{8z}+\frac{1}{16}\frac{(y-z)^2}{xyz}\right),~i = 5,6
\end{align*}
for $a_i \in \{0,1\}$ and $b_i, c_i \in \{0,1,2\}$ satisfying $a_i + b_i + c_i = d\leq 5$.

 As it can be checked, $$\Ricci_d\left(\sum_{i=1}^6x^iX_i\right) = \sum_{i=1}^6(x^i)^2\Ricci_d(X_i),$$ so that to ensure the existence of some $1\leq d \leq 5$ with positive $\Ricci_d$ curvature, it suffices to find such a $d$ constrained as: for any $a_i\in \{0,1\},~ b_i, c_i \in \{0,1,2\}$ with $a_i+b_i+c_i = d$
 it holds $\Ricci_d(X_i) > 0$ for some $(x, y, z) = \ga$. Therein we take $z = 1 - x - y$. Taking in account Equation \eqref{eq-compon-metr}, we abuse the notation and denote an invariant metric on $\mathrm{SU}(3)/\mathrm{T}^2$ by $\ga=(x,y,z)$. Next, we focus on \emph{Riemannian submersion metrics} since it suffices for our purposes, considering $x = y = t,~z = 1-2t$. 
\begin{proposition}\label{prop:riccinterd}
    Any invariant metric $\ga = (t,t,1-2t)$ in $\mathrm{SU}(3)/\mathrm{T}^2$ satisfies
    \begin{equation*}
  \Ricci_{d}>0,~d = 1, 2, 3, 4~\text{for}~ t\in ]\tfrac{3}{10}, \tfrac{1}{2}[
\end{equation*}
and for $d=1, 2, 3$ and $t=\tfrac{3}{10}$ there exists a non-zero $X\in T_o\mathrm{SU}(3)/\mathrm{T}^2$ such that $\Ricci_d(X) = 0$.
\end{proposition}
\begin{proof}

Considering the simplification imposed by $x = y = t$ (and so $z = 1-2t$) and normalizing appropriately since we are only considering $t < 1/2$ (because we need $z > 0$) one gets up multiplying by a positive function
\begin{align*}
  \Ricci_d(X_i) &= ((1{-}2 t) (-3 b_i + c_i + 2 (8a_i + 5 b_i - c_i) t)),~i=1,2&\\
  \Ricci_d(X_i) &= ((1{-}2 t) (-3 b_i + c_i + 2 (8a_i + 5 b_i - c_i) t)),~i=3,4 &\\
 \Ricci_d(X_i) &= b_i + c_i - 4 (b_i + c_i) t + 4 (4 a_i + b_i + c_i) t^2,~i=5,6
\end{align*}
for any $a_i\in \{0,1\},~ b_i, c_i \in \{0,1,2\}$ with $a_i+b_i+c_i = d$. Defining 
\begin{align*}
    (1-2t)^{-1}\Ricci_d(X_i) &:= f_d(t) = -3 b_i + c_i + 2 (8a_i + 5 b_i - c_i)t,~i=1,2,3,4&\\
    \Ricci_d(X_i) &:= g_d(t) = (d-a_i)(1-4t)+4(2a_i+d)t^2,~i=5,6,
\end{align*}
by directly varying the possibilities for $a_i, b_i, c_i$ assuming $8a_i+5b_i-c_i>0$ one gets that $f_d, g_d$ are positive simultaneously for every $t\in ]\frac{3}{10},\frac{1}{2}[$.

Now we consider $8a_i+5b_i \leq c_i$. Since $c_i \leq 2$ for every $i$, we get a contradiction if $a_i$ or $b_i$ are non-zero. Therefore, we must assume $a_i = b_i = 0$ and so $c_i = d \leq 2$. We get
\begin{align*}
    f_d(t) &= -2dt+d\\
    g_d(t) &= -4d+8dt
\end{align*}
and so $f_d(t) > 0$ if, and only if, $t<\tfrac{1}{2}$ and $g_d(t) > 0$ if, and only if, $t>\tfrac{1}{4}$. Hence, $f_d,~g_d$ are positive for $t\in ]\tfrac{1}{4},\tfrac{1}{2}[.$ Combining the information for $8a_i+5b_i > c_i,~8a_i+5b_i \leq c_i$ we get that for $d=1,2,3,4$ we have intermediate positive Ricci curvature for the metric $\ga = (t,t,1-2t)$ for any $t\in ]\tfrac{3}{10},\tfrac{1}{2}[$.

Finally, for $t=\tfrac{3}{10}$ and $d< 3$, picking $a_i=c_i=0$ we have $b_i=d$ and get
\[f_d(\tfrac{3}{10}) = 0.\]
For $d=3$ we pick $a_i=0, b_i=2, c_i=1$ so
\[f_d(\tfrac{3}{10})=0. \qedhere\]
\end{proof}

\section{The Ricci flow on $\mathrm{SU}(3)/\mathrm{T}^2$ does not preserve positive Ricci, sectional, and some intermediate Ricci curvatures}

\subsection{The Ricci flow on $\mathrm{SU}(3)/\mathrm{T}^2$}

In their work (\cite{BW}, Remark 3.2), B\"ohm and Wilking claim that there exists an invariant metric with positive Ricci curvature on the flag manifold $\mathrm{SU}(3)/\mathrm{T}^2$ that evolves under the homogeneous Ricci flow to a metric with mixed Ricci curvature. In this section, we provide explicit examples of invariant metrics that satisfy B\"ohm and Wilking claim for the \emph{backward homogeneous Ricci flow} instead by analyzing the global behavior of the homogeneous Ricci flow of $\mathrm{SU}(3)/\mathrm{T}^2$. We carry out this analysis using the {\em projected Ricci flow}, which was recently introduced in \cite{proj-ricci-flow}. Before doing so, we make a quick recall (aiming self-containing) on the analyses developed in \cite{proj-ricci-flow}.

\subsubsection{The very basics on the homogeneous Ricci flow on flag manifolds}
We recall that a family of Riemannian metrics $\mathsf{g}(t)$ in $M$ is called a Ricci flow if it satisfies
\begin{equation}
\label{ricci-flow}
\frac{\partial \mathsf{g}}{ \partial t}=-2\Ricci(\mathsf{g}).
\end{equation}

For any compact $n$-dimensional homogeneous space $M^n=G/K$ with connected isotropy subgroup $K$, a $G$-invariant metric $\ga$ on $M$ is determined by its value $\mathsf{g}_o$ at the origin $o=eK$, which is a $\Ad_G(K)$-invariant inner product. Just like $\ga$, the  Ricci tensor $\Ricci (\ga)$ and the scalar curvature $\mathrm{S}(\ga)$ are also $G$-invariant and completely determined by their values at $o$, $\Ricci(\ga)_o = \Ricci(\ga_o)$, $\mathrm{S}(\ga)_o=\mathrm{S}(\ga_o)$. Taking this into account, the Ricci flow equation (\ref{ricci-flow}) becomes the autonomous ordinary differential equation known as the (non-normalized) \textit{homogeneous Ricci flow}:
\begin{equation}
\label{inv-ricci-flow}
\frac{d\ga_t}{dt}=-2\Ricci(\ga_t).
\end{equation}

Recalling equation \eqref{eq-compon-metr} and \eqref{eq-compon-ricci} one derives that the Ricci flow (\ref{inv-ricci-flow}) becomes the autonomous system of ordinary differential equations
\begin{equation}
\label{eq-ricci-flow-coords}
\frac{dx_k}{dt}= -2 y_k, \qquad \ k=1,\ldots , n.
\end{equation}

It is always very convenient to re-write the Ricci flow equation in terms of the {Ricci operator} $r(\ga)_t$, which is possible since $r(\ga)_t$ is invariant under the isotropy representation and hence
$r(\ga)_t|_{\mathfrak{m}_k}$ is a multiple $r_k$ of the identity. From  \eqref{eq-compon-metr} and  \eqref{eq-compon-ricci}, we get
$$
y_k = x_k r_k
$$
and equation \eqref{eq-ricci-flow-coords} becomes
\begin{equation}
\label{eq-ricci-flow-final}
\frac{dx_k}{dt}= -2 x_kr_k.
\end{equation}

Following \cite{bwz}, we could also consider the flow below, which preserves the metrics with unit volume and is the gradient flow of $\ga\mapsto \mathrm{S}(\ga)$ when restricted to such space:
\begin{equation} \label{normaliz-ricci-flow}
 \frac{d\ga_t}{dt}=-2\left(\Ricci(\ga_t) - \frac{\mathrm{S}(\ga_t) }{n}{\ga}_t \right)
 \end{equation}
 As a variation of the above, we can re-normalize equation \eqref{inv-ricci-flow} by choosing a hypersurface in the (finite-dimensional) space of homogeneous metrics, which is transversal to the semi-lines $\lambda\mapsto \lambda \ga$. In \eqref{normaliz-ricci-flow}, the hypersurface consists of unit volume metrics and is unbounded. 

  Denoting by $R(x_1, \ldots, x_n)$ the vector field on the right-hand side of \eqref{eq-ricci-flow-final}, with phase space $\mathbb{R}_+^n = \{ (x_1, \ldots, x_n) \in  \mathbb{R}: \, x_i > 0 \}$, one can check that $x \in \mathbb{R}_+^n$ corresponds to an  Einstein metric if, and only if, $R(x) = \lambda x$, for some $\lambda > 0$. Furthermore, the homogeneous gradient Ricci flow (\ref{normaliz-ricci-flow}) on invariant metrics then becomes
\begin{equation}
\label{eq-ricci-flow-normaliz-final}
\frac{dx_k}{dt}= -2 x_kr_k - \frac{2}{n} S(x) x_k
\end{equation}
where 
\begin{equation}
\label{eq-scalar-curvature}
S(x) =  \sum_{i=1}^{n} n_i r_i
\end{equation}
is the scalar curvature, $n_k = \dim \lie m_k$.

	\begin{theorem}[Theorem 4.1 in \cite{proj-ricci-flow}]
		\label{thm-rescaling}
		For $x \in \mathbb{R}^n_+$, let $R(x)$ be a vector field, homogeneous of degree $0$ in $x$, and $W(x)$ a positive scalar function, homogeneous of degree $\alpha \neq 0$ in $x$. Suppose that $R(x)$ and $\rho(x) = W'(x) R(x)/\alpha$ are of class $C^1$. 
		Then, the solutions of
		\begin{equation}
		\label{eq-R1}
		\frac{d x }{dt}= R(x)
		\end{equation}
		can be rescaled in space and positively reparametrized in time to solutions of the normalized flow
		\begin{equation}
		\label{eq-R-normalizada}
		\frac{d x }{dt}= R(x) - \rho(x) x, \qquad W(x) = 1
		\end{equation}
		and vice-versa. Furthermore, $R(x) = \lambda x$ with $\lambda \in \mathbb{R}$ and $W(x) = 1$ if, and only if, $x$ is an equilibrium of equation (\ref{eq-R-normalizada}).
	\end{theorem}
  Such a result is introduced aiming to obtain the limit behavior of the homogeneous Ricci flow exploiting the rationality of $R(x)$. Hence, following their work, we normalize the homogeneous Ricci flow to a simplex and rescale it to get a polynomial vector field. More precisely, denoting by $ \overline x = x_1 + \ldots + x_n$, consider $W(x):= \overline x$ whose level set $W(x) = 1$ in $\mathbb{R}_+^n$ is the open canonical $(n-1)$-dimensional simplex, which is a bounded level hypersurface, in contrast with the unbounded unit-volume hypersurface. 
	
\begin{corollary}[Corollary 4.3 in \cite{proj-ricci-flow}]
		\label{corol-rescaling}
		The solutions of the Ricci flow
		\begin{equation}
		\label{eq-R}
		\frac{d x }{dt}= R(x)
		\end{equation}
		can be rescaled in space and reparametrized in time to solutions of the normalized flow
		\begin{equation}
		\label{eq-R-projetada}
		\frac{d x }{dt}=  R(x) - \overline{R(x)}x, \qquad \overline{x} = 1
		\end{equation}
		and vice-versa, where $x$ is Einstein with $\overline{x} = 1$ if and only if it is an equilibrium of equation (\ref{eq-R-projetada}).
		
		Moreover, there exists a function that is strictly decreasing on non-equilibrium solutions of the normalized flow \eqref{eq-R-projetada}. In particular, the projected Ricci flow does not have non-trivial periodic orbits.
	\end{corollary}

 Before proceeding, it is important to stress 
  \begin{remark}[The fixed points for the flow] \label{rem:einstein}
    A very important property is that, in our case,  
    the Einstein metrics are precisely the fixed points of the flow of the Ricci system (singularities of the Ricci system) and, from the point of view of dynamics in the plane, no flow line attracted to an Einstein metric achieves it in finite time due to the qualitative nature of these singularities.
\end{remark}

\subsubsection{The projected homogeneous Ricci flow in $\mathrm{SU}(3)/\mathrm{T}^2$}
\label{sec:thericciflowhere}
Recall the isotropy representation of $\mathrm{SU}(3)/\mathrm{T}^2$ decomposes into three irreducible and non-equivalent components:
$$
\mathfrak{m}=\mathfrak{m}_1 \oplus \mathfrak{m}_2 \oplus \mathfrak{m}_3.$$

 The Ricci tensor of an invariant metric $\ga = (x,y,z)$ is also invariant, and its components are given by (recall equations \eqref{eq:ricci01}-\eqref{eq:ricci02}):
 
\begin{eqnarray*}
r_x&=&\frac{1}{2 x} +\frac{1}{12} \left(\frac{x}{y z}-\frac{z}{x y}-\frac{y}{x z}\right),\\ \\
r_y&=&\frac{1}{2 y} + \frac{1}{12} \left(-\frac{x}{y z}-\frac{z}{x y}+\frac{y}{x z}\right), \\ \\ 
r_z&=&\frac{1}{2 z}+ \frac{1}{12} \left(-\frac{x}{y z}+\frac{z}{x y}-\frac{y}{x z}\right)
\end{eqnarray*}
and the corresponding (unnormalized) Ricci flow equation is given by 
\begin{equation}\label{RF-eq}
 x^\prime = -2x r_x, \quad y^\prime = -2y r_y, \quad z^\prime = -2z r_z.   
\end{equation}

The projected Ricci flow is obtained by a suitable reparametrization of the time, getting an induced system of ODEs with phase-portrait on the set
$$
\{(x,y,z)\in \mathbb{R}^3: x+y+z=1  \}\cap {\bb R}^3_+,
$$
where $\mathbb{R}^3_+=\{ (x,y,z)\in \mathbb{R}^3: x>0,y>0,z>0  \}$, following of the projection on the $xy-$plane. The resulting system of ODEs is dynamically equivalent to the system \eqref{RF-eq} (Corollary \ref{corol-rescaling}).

Applying the analysis developed in \cite[Section 5]{proj-ricci-flow}, we arrive at the equations of the projected Ricci flow equation (see equation (31) in Section 5 of \cite{proj-ricci-flow}):
\begin{equation}\label{proj-ricci-flow}
    \left\{\begin{array}{l}
    x^{\,\prime} = {u(x,y)},\\
    y^{\,\prime} = {v(x,y)},
    \end{array}\right.
\end{equation}
where $$u(x,y)=2 x \left(x^2 (2-12 y)-3 x \left(4 y^2-6 y+1\right)+6 y^2-6 y+1\right),$$ and $$v(x,y)=-2 y (2 y-1) \left(6 x^2+6 x (y-1)-y+1\right).$$ 

The global dynamics of \eqref{proj-ricci-flow} is described in Figure \ref{fig:fig662}. The remarkable equilibrium points in the projected Ricci flow are (see Theorem 5.1 in \cite{proj-ricci-flow}):
\begin{itemize}
\item $A=(1/4,1/4)$: K\"ahler Einstein metric (saddle),
    \item $B=(1/3,1/3)$: the normal Einstein metric (repeller),
    \item $C=(1/4,1/2)$: K\"ahler Einstein metric (saddle),
    \item $D=(1/2,1/4)$: K\"ahler Einstein metric (saddle).
\end{itemize}

\begin{center}
\begin{figure}
    \centering
    \includegraphics[height=7.5cm]{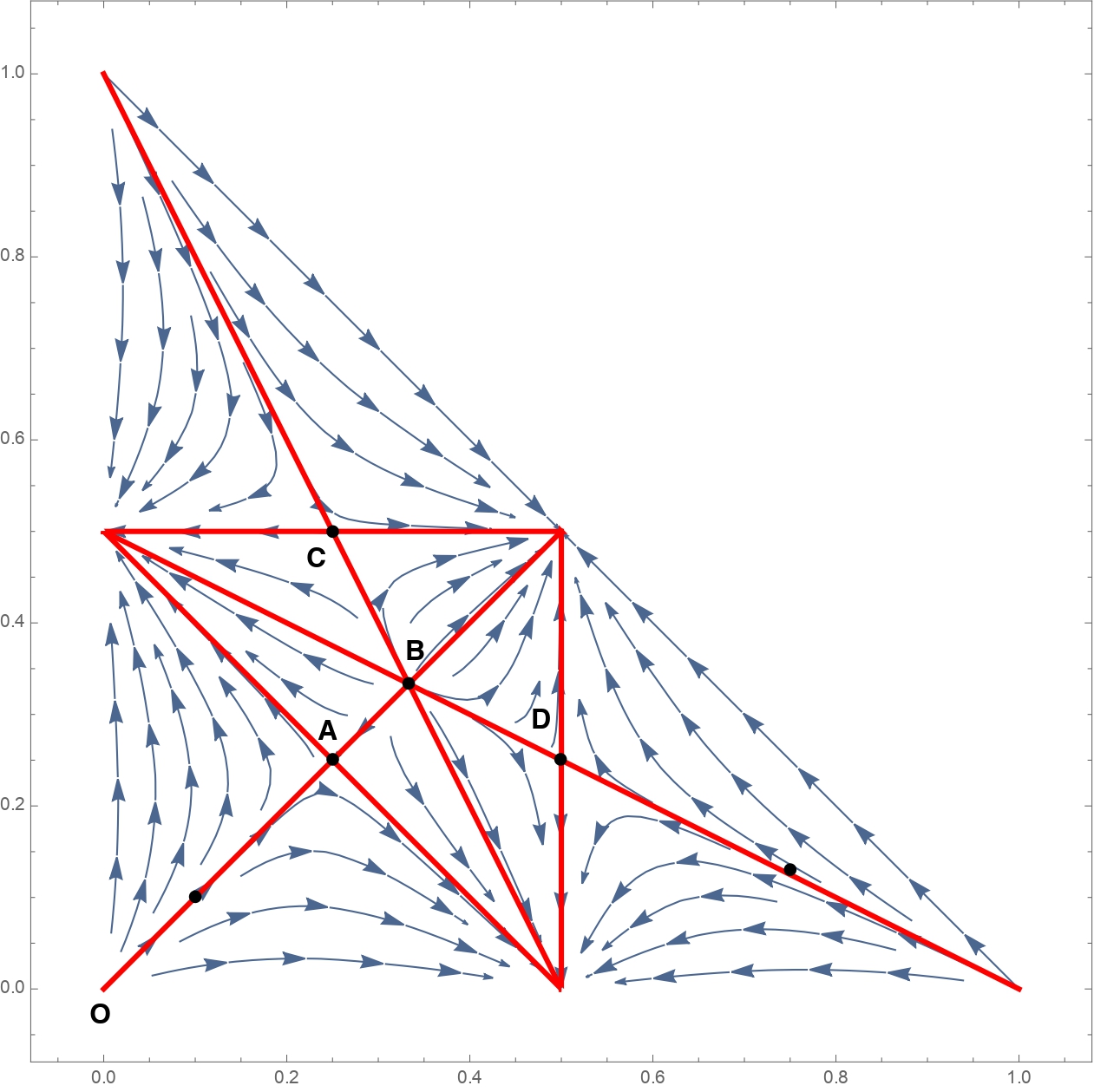}
    \caption{Phase portrait of the Projected Ricci Flow for $\mathrm{SU}(3)/\mathrm{T}^2$.}
    \label{fig:fig662}
\end{figure}
\end{center}

To proceed with our analysis, we start with the following remark:
\begin{lemma}\label{lem:invariantsegment}
  The segment $OA$, where $O=(0,0)$, is invariant by the flow of \eqref{proj-ricci-flow}. Hence, the segment $OA$ solves the projected Ricci flow.   
\end{lemma}
\begin{proof}
    The segment $OA$ is supported in the line $\gamma(t)=(t,t)$ with normal vector $w=(-1,1)$. We have
    \begin{eqnarray*}
        \left(u(\gamma(t)),v(\gamma(t))\right)\cdot w &=& (-1)(2 t \left(-24 t^3+26 t^2-9 t+1\right)) + (1)(2 t \left(-24 t^3+26 t^2-9 t+1\right))\\
        &=& 0,
    \end{eqnarray*}
    and therefore the segment $OA$ is solution of the system \eqref{proj-ricci-flow}.
\end{proof}

Our second remark is about the Ricci curvature under the segment $OA$. 
\begin{lemma}\label{lema-sinal-ricci}
    The segment $OA$ contains invariant metrics with positive Ricci curvatures and invariant metrics with mixed Ricci curvatures.
\end{lemma}
\begin{proof}
    Since the flow is projected in the $xy$-plane, we start to lift the segment $\gamma(t)=(t,t)$, $0<t<1/2$, to the plane $x+y+z=1$. The corresponding segment is parametrized by $\Tilde{\gamma}(t)=(t,t,1-2t)$. The components of the Ricci tensor along the segment $\Tilde{\gamma}$ are given by:
    \begin{eqnarray*}
        r_x(\Tilde{\gamma}(t))&=&\frac{1}{2 t}-\frac{1-2 t}{12 t^2},  \\
        r_y(\Tilde{\gamma}(t))&=&\frac{1}{2 t}-\frac{1-2 t}{12 t^2},  \\
        r_z(\Tilde{\gamma}(t))&=& \frac{1}{12} \left(\frac{1-2 t}{t^2}-\frac{2}{1-2 t}\right)+\frac{1}{2 (1-2 t)}.
    \end{eqnarray*}
    A straightforward computation shows that $r_z(\Tilde{\gamma}(t))>0$ for $0<t<1/2$. We also have $r_x(\Tilde{\gamma}(t))>0$, $r_y(\Tilde{\gamma}(t))>0$ for $1/8<t<1/2$, and $r_x(\Tilde{\gamma}(t))<0$, $r_y(\Tilde{\gamma}(t))<0$ for $0<t<1/8$   
\end{proof}

\begin{remark}
One can rescale the family of metrics $(t,t,1-2t)$ in order to obtain the family of metrics $(1,1,\frac{1-2t}{t})$, $0<t<1/2$. Such a family appears as deformation of the normal metric $(1,1,1)$ in the direction of fibers of the homogeneous fibration $\mathbb{S}^2\to \mathrm{SU}(3)/\mathrm{T}^2\to \mathbb{CP}^2$. 
\end{remark}

\begin{theorem}\label{thm:basiccase}
    There exists an invariant metric on $\mathrm{SU}(3)/\mathrm{T}^2$ with positive Ricci curvature that evolves into an invariant metric of mixed Ricci curvature under the (backward) homogeneous Ricci flow. 
\end{theorem}
\begin{proof}
    The natural choice for the initial metric $\ga_0$ with positive Ricci curvature is the K\"ahler-Einstein metric $(1/4,1/4)$ ($\sim \,$ metric determined by $ (1,1,2)$). However, since this metric is a fixed point of the system \eqref{proj-ricci-flow} (Remark \ref{rem:einstein}), we start with a metric that is arbitrarily close to the K\"ahler-Einstein metric, namely $\ga_0=(t_0,t_0,1-2t_0)$, where $1/8<t_0<1/4$. According to Lemma \ref{lema-sinal-ricci}, we have $\Ricci \ga_0 >0$. The solution of the projected Ricci flow with initial condition $\ga_0$ is contained in the segment $OA$, and it is attracted to the K\"ahler-Einstein metric. On the other hand, if we consider the backward solution, by continuity, there exists $T>-\infty$ such that $\ga_T=(t_T,t_T,1-2t_T)$ with $t_T<1/8$ has mixed Ricci curvature, by Lemma \ref{lema-sinal-ricci}.
\end{proof}

\begin{remark}[Some metrics with positive Ricci curvature do not evolve to mixed Ricci curvature metric]
    In complete analogy to Theorem C in \cite{BW}, one could ask whether any metric with positive Ricci curvature on $\mathrm{SU}(3)/\mathrm{T}^2$ develops Ricci curvature with a mixed sign under the Ricci flow. One observes that taking the path of metrics $\ga = (\h,t,\h-t)$ we have that it corresponds to a solution for the projected Ricci flow for which through all $t\in ]0,\h[$ the path has positive Ricci curvature. This can be easily checked by observing that $(1,0)\cdot (u(\h,t),v(\h,t)) = 0$, so the path yields a solution for the projected Ricci flow, and using the explicit equations for $r_x, r_y, r_z$ along $\Tilde\gamma(t):= (\h,t,\h-t)$ to verify the claim.
\end{remark}

Observe now that considering the normal metric $(1,1,1)$, we have positive sectional curvature on both fiber and base of $S^2\rightarrow \mathrm{SU}(3)/\mathrm{T}^2\rightarrow \mathbb{CP}^2$. The family of metrics $(1,1,(1-2t)/t),~t\in ]0,1/2[$ is a canonical variation of $(1,1,1)$. According to Theorem A in \cite{reiser2022positive}, one can find $\tau > 0$ such that for any $t\leq \tau$, we have $\Ricci_{d} > 0$ for any $d\geq 5$, but this is just ordinary positive Ricci curvature. In this manner,  Theorem A in \cite{reiser2022positive}, despite the fascinating examples built of manifolds with intermediate positive Ricci curvature, cannot recover Proposition \ref{prop:riccinterd}. We point out, however, that Theorem A in \cite{reiser2022intermediate} is sharp in the sense that, under their very general hypotheses, one cannot hope to reach a stronger conclusion. In particular, if one considers a trivial bundle, a canonical variation does not improve the intermediate Ricci curvature, meaning one only has $\Ricci_d > 0$ for the values of $d$ given in their Theorem A. Our result illustrates that under certain conditions, one can find metrics on bundles that satisfy stronger positive intermediate curvature conditions than the one given in Theorem A of \cite{reiser2022intermediate}.\footnote{We kindly thank L. Mouillé for pointing it out.}

\begin{theorem}\label{thm:basiccase2}
There exists an invariant metric in $\mathrm{SU}(3)/\mathrm{T}^2$ with $\Ricci_d > 0$ for $d = 1, 2, 3$, that evolves into an invariant metric for such $\Ricci_d(X) \leq 0$ for some non-zero $X$ under the homogeneous Ricci flow. 
\end{theorem}
\begin{proof}
Let us consider the line segment $(t,t)$ with $t\in ]0,\tfrac{1}{4}[$. Such a segment is invariant for the flow and consists of a complete solution not achieving the K\"ahler-Einstein metric $(1/4,1/4)$ in finite time. Analogously, the same segment for $t\in ]\tfrac{1}{4},\tfrac{1}{3}[$ is invariant by the flow and connect (but not in finite time) the K\"ahler-Einstein metric $A = (1/4,1/4)$ and the normal Einstein metric $B = (1/3,1/3)$.

For $d= 1, 2, 3$ let us take $t_0 \in ]\tfrac{5}{16},\tfrac{1}{3}[$ and $\ga = (t_0,t_0,1-2t_0)$. Observe that such a metric has positive $\Ricci_{d} > 0$ for $d = 1, 2$, accordingly to Proposition \ref{prop:riccinterd}. Since the segment $(t_0,t_0)$ is a solution for the flow for $]1/4,t_0]$ such that its limit for infinite time is the K\"ahler-Einstein metric given by $(1/4,1/4)$; that do not have positive sectional curvature, neither $\Ricci_2, \Ricci_3 > 0$ (Proposition \ref{prop:riccinterd}), one concludes the claim since continuity ensures that starting the Ricci flow for such chosen $t_0$ will lead in future time (that can be chosen to be finite) to a metric close enough in the $\mathrm{C}^2$-topology to such a K\"ahler-Einstein metric, thus not having $\Ricci_2, \Ricci_3 > 0$ nor $\Ricci_1 > 0$. It is worth mentioning that the backward Ricci flow does maintain such properties since this metric is attracted to $B$.  \qedhere
\end{proof}

\section{Family $\mathrm{SU}(m+2p)/\mathrm{S}(\mathrm{U}(m)\times \mathrm{U}(p) \times \mathrm{U}(p))$, with $m\geq p >0$}

One central motivation for this note is the results in \cite{BW}. Their Theorem C shows that the compact manifold $M = \mathrm{Sp}(3)/\mathrm{Sp}(1)\times \mathrm{Sp}(1)\times \mathrm{Sp}(1)$ evolves a specific positively curved metric in a metric with mixed Ricci curvature. Such a space is a homogeneous manifold with two parameters describing homogeneous unit volume metrics. The invariant metric of $\mathrm{Sp}(3)$ induces on $M$ a homogeneous unit volume Einstein metric $\ga_E$ of non-negative sectional curvature.

Associated to such homogeneous space, one can consider the fibration $\mathrm{S}^4\hookrightarrow M\rightarrow \bb{H}P^2$ and combining a \emph{canonical variation} on the Riemannian submersion metric obtained out $\ga_E$ plus scaling; we obtain a curve $\ga_t,~t>1$ of unit volume submersion metrics with positive sectional curvature for which $\ga_0 = \ga_E$.
 
 Up to parametrization, such a curve is a solution to the normalized Ricci flow. A precise analysis of the asymptotic behavior of solutions of the Ricci flow allows the authors to prove for any homogeneous non-submersion initial metric, being close enough to $\ga_2$,
the normalized Ricci flow evolves mixed Ricci curvature.

Considering the former paragraphs, one observes that the analyses employed in this note are similar. We look for particular Einstein metrics and observe the long-time behavior of some parametrized solutions to the (projected) homogeneous Ricci flow. Hence, as a final aim, we furnish similar results to Theorem C in \cite{BW} carrying on the analyses made for $\mathrm{SU}(3)/\mathrm{T}^2$ to a family which generalizes it: The ones in $\mathrm{SU}(m+2p)/\mathrm{S}(\mathrm{U}(m)\times \mathrm{U}(p) \times \mathrm{U}(p))$, with $m\geq p >0$. 

As the Ricci operator $r(\ga)$ is diagonalizable, one can infer from \cite{sakane99} that its eigenvalues are of multiplicity two and given by
$$
\begin{array}{lcl}
\vspace{0.3cm}r_x&=&\displaystyle \frac{1}{2 x}+\frac{p}{4 (m+2 p)} \left(\frac{x}{y z}-\frac{z}{x y}-\frac{y}{x z}\right),\\
\vspace{0.3cm}r_y&=&\displaystyle\frac{1}{2 y} + \frac{p}{4 (m+2 p)} \left(-\frac{x}{y z}-\frac{z}{x y}+\frac{y}{x z}\right),\\
\vspace{0.3cm}r_z&=&\displaystyle\frac{1}{2 z}+\frac{m}{4 (m+2 p)} \left(-\frac{x}{y z}+\frac{z}{x y}-\frac{y}{x z}\right).
\end{array}
$$
We intend to apply the concept provided by Definition \ref{def:minimun}. Namely, we prove
\begin{theorem}\label{thm:generalfamily}
    For each $t\in ]1/8,1/3[$ the Riemannian metric $\ga = (t,t,2-t)$ has $d$-positive Ricci tensor (Definition \ref{def:minimun}) for $d\in \left\{1,\ldots,4mp + 2p^2\right\}$. Moreover, the backward homogeneous Ricci flow loses this property in finite time.
\end{theorem}
With this aim, we proceed once more by describing the projected homogeneous Ricci flow as (see equation (31) in \cite[Section 5]{proj-ricci-flow})
$$
\begin{array}{lcl}
\vspace{0.5em}u(x,y)&=&-x (2 x-1) \left(m (4 y-1) (x+y-1)+p \left(x (8 y-1)+8 y^2-7 y+1\right)\right),\\
v(x,y)&=&-y (2 y-1) \left(m (4 x-1) (x+y-1)+p \left(8 x^2+x (8 y-7)-y+1\right)\right).
\end{array}
$$

\begin{lemma}
Let $K=\left(\frac{m+p}{2 (m+2 p)},\frac{m+p}{2 (m+2 p)} \right)$. The segment $OK$ is invariant under the projected Ricci flow. 
\end{lemma}
\begin{proof}
    We need compute $(u(\gamma(t)),v(\gamma(t)))\cdot w$ where $\gamma(t)=(t,t)$ and $w=(-1,1)$. Straightforward computation shows: 
    $$
    u(t,t)=v(t,t)=-t (1 - 6 t + 8 t^2) (m (-1 + 2 t) + p (-1 + 4 t)).
    $$
   Hence, 
\begin{align*}
   (-1,1)\cdot \left(u(\gamma(t)),v(\gamma(t))\right) &= -u(\gamma(t)) + v(\gamma(t))\\
   &= -u(\gamma(t)) + u(\gamma(t))\\
   &= 0. \qedhere
    \end{align*}
\end{proof}

Now consider the lifted curve $\Tilde{\gamma}(t)=(t,t,1-2t)$. We have:

\begin{align}
\label{eq:riccipositivity1} r_1(\Tilde{\gamma(t)})&=\frac{1}{2 t}-\frac{p (1-2 t)}{4 t^2 (m+2 p)},\\
\label{eq:riccipositivity2} r_2(\Tilde{\gamma(t)})&=\frac{1}{2 t}-\frac{p (1-2 t)}{4 t^2 (m+2 p)},\\
\label{eq:riccipositivity3} r_3(\Tilde{\gamma(t)})&= \frac{1}{2 (1-2 t)}+\frac{m}{4 (m+2 p)} \left(\frac{1-2 t}{t^2}-\frac{2}{1-2 t}\right)
\end{align}

Straightforward computation guarantees

$$
\begin{array}{l}
\vspace{0.3cm}r_x(\Tilde{\gamma(t)})>0 \Leftrightarrow t>\dfrac{p}{2 m+6 p},\\
\vspace{0.3cm}r_y(\Tilde{\gamma(t)})>0 \Leftrightarrow t>\dfrac{p}{2 m+6 p},\\
\vspace{0.3cm}r_z(\Tilde{\gamma(t)})>0 \Leftrightarrow 0< t< \dfrac{1}{2}.
\end{array}
$$
We claim that for each $t\in ]1/8,1/3[$ we have $d$-positivity of the Ricci tensor (Definition \ref{def:minimun}) for $d = 1,\ldots, 4mp + 2p^2$ where 
\begin{equation}
4mp + 2p^2 = \dim \mathrm{SU}(m+2p)/\mathrm{S}(\mathrm{U}(m)\times\mathrm{U}(p)\times\mathrm{U}(p)).
\end{equation}
\begin{lemma}\label{lem:region1}
    For each $d\in \{1,\ldots,\dim \mathrm{SU}(m+2p)/\mathrm{S}(\mathrm{U}(m)\times\mathrm{U}(p)\times\mathrm{U}(p))\}$ the Ricci tensor $r(\ga)$ of a metric $\ga = (t,t,1-2t)$ is $d$-positive if $t\in ]1/8,1/3[$. In particular, $t$ does not depend on $(m,p)$, which holds for the entire considered family.
\end{lemma}
\begin{proof}
    Since $\frac{m}{p}\geq 1$ one has that 
    \begin{align*}
    \frac{p}{2m+6p} &= \frac{1}{2\left(\tfrac{m}{p}\right)+6}\\
    &\leq \frac{1}{8}.
    \end{align*}
Therefore, according to equations \eqref{eq:riccipositivity1},\eqref{eq:riccipositivity2}, one gets that tf $t> \frac{1}{8}$ then $r_1,r_2 > 0$. On the other hand, $t\in ]0,1/2[$ ensures that $r_3 > 0$ (equation \eqref{eq:riccipositivity3}). Observe, however, that geometrically, we cannot attach the metric for $t\geq 1/3$ since one arises in a KE metric for such a parameter, and that is a fixed point of the system. In this manner, the claim follows for $t\in ]1/8,1/3[$.
    \qedhere
\end{proof}

We finish proving Theorem \ref{thm:generalfamily}.
\begin{proof}[Proof of Theorem \ref{thm:generalfamily}]
We proceed exactly as in the proof of Theorem \ref{thm:basiccase} observing that, similarly to that case, we could be tempted to consider as initial metric the K\"ahler-Einstein metric (on the projected Ricci-flow) given by
$$K=\left(\frac{m+p}{2 (m+2 p)},\frac{m+p}{2 (m+2 p)} \right).$$ Note however that since $\dfrac{m}{p}\geq 1$ we have 
\begin{align*}
    \frac{m+p}{2 (m+2 p)} &= \frac{m+p}{2(m+p)+2p}\\
    &= \frac{1}{2+\tfrac{2p}{m+p}}\\
    &= \frac{1}{2+\tfrac{2}{(m/p)+1}}\\
    &\geq \frac{1}{3}.
\end{align*}

Hence, for $t\in ]1/8,1/3[$, we have that, according to Lemma \ref{lem:region1}, the desired $d$-positivity condition to the Ricci tensor. Moreover, starting the homogeneous backward Ricci flow, the result is ensured by continuity. \qedhere
\end{proof}

\begin{remark}\label{rem:final}
Lastly, we recall that $r(\ga)$ losing $d$-positivity for $d=4mp + 2p^2$ implies losing positive scalar curvature. Hence, direct but cumbersome computations coming from \cite{Goette1999} make it possible to recover that the space of moduli of invariant metrics of positive Ricci curvature on $\mathrm{SU}(m+2p)/\mathrm{S}(\mathrm{U}(m)\times\mathrm{U}(p)\times\mathrm{U}(p))$ has infinitely many path components. This shall appear elsewhere.
\end{remark}

\section*{Acknowledgements}
This paper was written for the Conference ``GS\&MS Meeting proceedings from August 29th to September 2nd, 2022''. The authors gladly acknowledge the organizers who promoted an excellent cooperation and mathematics learning environment. We extend our gratitude to the hospitality of the Universidad Nacional de C\'ordoba. L. F. C. also thanks FAEPEX n. 519.292--2449/22 for funding his participation at the conference.
The São Paulo Research Foundation (FAPESP) supports L. F. C. grant 2022/09603-9, partially supports L. G. grants 2021/04003-0, 2021/04065-6, and partially supports R. M. M. grants 2021/08031-9, 2018/03338-6. The National Council for Scientific and Technological Development (CNPq) supports R. M. M. grants  315925/2021-3 and 434599/2018-2.

The authors gladly acknowledge the anonymous referees for the careful read, leading useful suggestions that improved the quality of the paper, and the editor for well handling the manuscript submission.

	\bibliographystyle{alpha}

	\bibliography{main}

\end{document}